\documentclass[12pt, a4paper]{article}
\usepackage[utf8]{inputenc}
\usepackage[T1]{fontenc}
\usepackage{lmodern}
\usepackage{graphicx,color}
\usepackage[left=3cm,right=3cm,top=3cm,bottom=2cm]{geometry}
\usepackage{amsthm,amsfonts,amsmath, mathrsfs}
\usepackage[colorlinks=true,linkcolor=blue,citecolor=magenta,hyperindex]{hyperref}
\usepackage{amssymb}
\usepackage{needspace}
\usepackage{verbatim}
\usepackage[all]{xy}
\usepackage{natbib}
\usepackage{titlesec}
\usepackage{tikz-cd}
\usepackage{float}
\usepackage{lipsum}
\usepackage{authblk}
\usepackage{mathtools}
\usepackage{subfig}
\titleformat{\section}[runin]  
  {\normalfont\bfseries}       
  {\thesection.}               
  {1em}                        
  {}                           
  [.]                          

\titlespacing*{\section}
  {0pt}{1em}{1em} 

  \titleformat{\subsection}[runin]  
  {\normalfont\bfseries}       
  {\thesubsection.}              
  {1em}                      
  {}                          
  [.]                          

\titlespacing*{\subsection}
  {0pt}{1em}{1em}           

\title{\textbf{\Large Anomalous cw-expansive homeomorphisms on compact surfaces of higher genus}}

\author{\textsc{\small Alberto Sarmiento, Douglas Danton and Viviane Pardini Valério}}

\date{\small\today}

\begin{document}

\maketitle

\vspace{-1.5cm}

\newtheorem{teorema}{Theorem}[section]
\newtheorem{proposicao}{Proposition}[section]
\newtheorem{lema}{Lemma}[section]
\newtheorem{corolario}{Corollary}[section]
\newtheorem{observacao}{Remark}
\newtheorem{exemplo}{Example}[subsection]
\newtheorem{definicao}{Definition}[section]

\renewcommand{\abstractname}{}
\begin{abstract}
\noindent {\small \textsc{Abstract:} In this paper, we construct cw-expansive homeomorphisms on compact surfaces of genus $g \ge 0$ with a fixed point whose local stable set is connected but not locally connected. This provides an affirmative answer to question posed by Artigue \cite{artigueanomalous}. To achieve this, we generalize the construction from the example of Artigue, Pacifico and Vieitez \cite{artiguepacificovieitez}, obtaining examples of homeomorphisms on compact surfaces of genus $g \ge 2$ that are 2-expansive but not expansive. On the sphere and the torus, we construct new examples of $cw2$-expansive homeomorphisms that are not $N$-expansive for all $N\geq 1$.}
\end{abstract}

\vspace{0.06cm}

\section{Introduction} \label{preliminaries} 

The concept of expansiveness was introduced by Utz \cite{utz} in the 1950s. Examples of expansive systems include Anosov diffeomorphisms and Axiom A diffeomorphisms restricted to the non-wandering set. The pioneer of this subject is Mañe \cite{mane1979} in 1979. In \cite{hiraide,lewowicz} Hiraide and Lewowicz proved that every expansive homeomorphism of a compact surface is conjugate with a pseudo-Anosov diffeomorphism. In 1993, Kato \cite{katocw} introduced a generalization of expansivity, the notion of continuum-wise ($cw$) expansive homeomorphisms. The concept of $N$-expansiveness was introduced by Morales \cite{morales} and the concept of $cwN$-expansive was introduced by Artigue \cite{dentritations}. 

Recall that a homeomorphism  $f: X \longrightarrow X$, where $X$ is a compact metric space, is called expansive, if there exists $\alpha > 0$ such that for any distinct points $x,y \in X$,where the set $d(f^n(x),f^n(y)) > \alpha$ for some $n \in \mathbb{N}$. Equivalently, $\Gamma_{\alpha}(x) = \{x\}$ for all $x \in X$ where the set $\Gamma_{\alpha}(x) = \{y \in X ; d(f^k(x),f^k(y)) \le \alpha, \forall k \in \mathbb{Z} \}$ is called \textit{dynamic ball}. The homeomorphism $f$ is said $N$-expansive if there exists $\alpha > 0$ such that the cardinality of the dynamic ball $\sharp(\Gamma_{\alpha}(x)) \le N$ for all $x \in X$ (In this sense, an expansive homeomorphism is $1$-expansive). A homeomorphism $f$ is called $cw$-expansive (\textit{continuum-wise} expansivity) if there exists $\alpha > 0$ such that for every continnum $C \subset X$, if $diam(f^k(C)) \leq \alpha$ for all $k \in \mathbb{Z}$, then $C$ is a singleton. A homeomorphism $f$ is $cwN$-expansive if there is $\alpha > 0$ such that if $A$, $B\subset X$ are continua,  $diam(f^k(A)) \leq \alpha$ for all $k \geq 0$ and  $diam(f^k(B)) \leq \alpha$ for all $k \leq 0$ then $\sharp(A\cap B)\leq N$.  In this situation, if $\sharp(A \cap B)$ is finite (independent of $N$), we say that $f$ is $cw_F$-expansive (where the $F$ means finite).

The motivation for this paper arose from Artigue's article \cite{artigueanomalous}, in which the following question is posed: Does every compact surface admit a cw-expansive homeomorphism with a connected but non-locally connected stable set? We answer this question affirmatively in Theorem (\ref{thprincipal}). The following diagram summarizes the definitions and presents a hierarchical structure of expansiveness, analogous to Table 1 in Artigue \cite{dentritations}, p. 5.

 $\small \xymatrix{
*+[F-:<3pt>]{\text{Exp}}  \ar@{<->}[r]  & *+[F-:<3pt>]{\text{$1$-exp}} \ar@/^0.8cm/[r] \ar[dd] & *+[F-:<3pt>]{\text{$2$-exp}} \ar@/^0.8cm/[r]  \ar[dd] & \dots \ar@/^0.8cm/[r] \ar[dd] & *+[F-:<3pt>]{\text{$N$-exp}} \ar[dd] \\
 & & & \\
 & *+[F-:<3pt>]{\text{$cw1$-exp}} \ar@/_0.8cm/[r] & *+[F-:<3pt>]{\text{$cw2$-exp}} \ar@/_0.8cm/[r] & \dots \ar@/_0.8cm/[r] & *+[F-:<3pt>]{\text{$cwN$-exp}} \ar[r]  & *+[F-:<3pt>]{\text{$cw_F$-exp}} \ar[r] & *+[F-:<3pt>]{\text{$cw$-exp}} }$

\vspace{1.2cm}

For the proof of this theorem, in Subsection (\ref{Gluing DA-plugs}), we generalize the construction of the example introduced by Smale \cite{smale1967}, namely Derived-from-Anosov diﬀeomorphism or the DA-diﬀeomorphism Theorem according Robinson \cite{robinsonbook}. To do this, we construct a stable DA-plug and an unstable DA-plug (Subsection \ref{DA-plug}). We proceed to glue a stable DA-plug with another unstable DA-plug, following the ideas of Artigue, Pacifico, and Vieitez \cite{artiguepacificovieitez}. We recall this construction and its properties in terms of gluing plugs. In Section (\ref{secaomainresults}) we show families of new examples on compact surfaces.

\begin{teorema}\label{2expansivos}
There exist 2-expansive homeomorphisms on compact surfaces of genus greater than or equal to 2 which are not expansive.
\end{teorema}

Using the same technique of gluing DA-plugs onto the  pseudo-Anosov diffeomorphism with 1-prongs on the sphere, such an example already appears in P. Walters (Example 1, p. 140), we were able to construct cw2-expansive examples on the sphere with an attracting set that are non-transitive (the set of non-wandering points is a proper subset of the surface) (Proposition \ref{cw2expansivos-0}). We also obtain examples on the torus that are cw2-expansive homeomorphisms with 1-prongs and are non-transitive (Proposition \ref{cw2expansivos-1}).

Finally, we defined the anomalous plug as in Artigue  \cite{artigueanomalous} and were able to insert it into each of the previous examples, obtaining the following theorem.

\begin{teorema} \label{thprincipal} Every compact surface ($g \ge 0$) admits a cw-expansive homeomorphism with a connected stable set but non-locally connected.
\end{teorema}

\section{Preliminary discussions} \label{preliminariesdiscussions}

In this Section, we review in Subsection (\ref{DA-plug}) the notion of DA-diffeomorphism and introduce the concepts of \textit{unstable DA-plug} and \textit{stable DA-plug}. In Subsection (\ref{Gluing DA-plugs}) we will proceed to glue two DA-plugs, always pairing one stable plug with one unstable plug. Finally, in Subsection (\ref{2-expansive}), we perform the connected sum of disjoint copies of $\mathbb{T}^2$, obtaining the bitorus and a surface homeomorphism.

First, we present some preliminary results relating the expansiveness properties of a homeomorphism to those of its iterates. It is known that the expansiveness property is preserved for a finite number of iterations, that is, the homeomorphism $f$ is expansive if and only if $f^k$ is expansive $(k \neq 0)$. The following result shows that this property holds true in the case of cw-expansive homeomorphisms defined on a compact metric space $X$.

\newpage

\begin{lema}\label{iteradacwexpansivos}
Let $f:X  \longrightarrow X$ be a homeomorphism. The following are equivalent:
\begin{enumerate}
    \item $f$ is $cw$-expansive;
    \item $f^k$ is $cw$-expansive for all $k\in {\mathbb{Z}}\setminus\{0\}$;
    \item $f^{i_0}$ is $cw$-expansive for some $i_0\in {\mathbb{Z}}\setminus\{0\}$.
\end{enumerate}
\end{lema}

\begin{proof}
($1 \Rightarrow 2$) Fixed  $k_0\in {\mathbb{Z}}\setminus\{0,~ 1\}$. Let $c>0$ the cw-expasivity constant for $f$, i.e. if $diam(f^n(A))<c$ for some continuous $A$ and for all $n\in{\mathbb{Z}}$, then $\sharp (A)=1$. 

Since homeomorphisms $\{f, f^2, \cdots, f^{k_0} \}$ are uniformly continuous, give $\epsilon=c/2>0$, there exist $\delta>0$ such that if $d(x,y)<\delta$ then 
\begin{equation}\label{uniforme}
d(f^i(x),f^i(y))<c/2, ~\mbox{for all}~ i=1,2,\cdots , k_0-1. 
\end{equation}

We claim that $\overline{c}=\min\{c, \delta\}$ is the cw-expansivity constant for $f^{k_0}$. In fact, suposse that $diam((f^{k_0})^n(A))<\overline{c}$ for some continuous $A$ and for all $n\in{\mathbb{Z}}$, so  
\begin{equation}
diam(f^{nk_0}(A))<\overline{c},~~\forall n\in{\mathbb{Z}}
\end{equation}

Then $d(f^{nk_0}(x),f^{nk_0}(y))<\overline{c}~~ (<\delta)$ for all $x,y\in A$ and for all $n\in{\mathbb{Z}}$, by \eqref{uniforme} we have: 
 $$d(f^i(f^{nk_0}(x)),f^i(f^{nk_0}(y)))<c/2, \forall x,y\in A,~~\mbox{for all}~ i=0,1,2,\cdots , k_0-1.$$ 

Then $diam(f^{nk_0+i}(A))\leq c/2<c$. Then $\sharp (A)=1$

($2 \Rightarrow 3$) is obvious. ($3 \Rightarrow 1$) indeed, suppose that $diam((f)^n(A))<\overline{c}$ for some continuous $A$, for $c$ constant of $cw$ - expansivity of $f^{i_0}$ and for all $n\in{\mathbb{Z}}$. In particular, we have that for the iterated multiples of $i_0$, thus we conclude that $\sharp (A)=1$.
\end{proof}

From the definition of $cwN$-expansive and following the proof of the previous lemma, we have:

\begin{lema}\label{iteradacw2expansivos}
Let $f:X \longrightarrow X$ homeomorphism, then $f$ is cwN-expansive if and only if $f^k$ is cwN-expansive for all $k \in \mathbb{Z} \setminus {0}$.
\end{lema} 

The previous lemma also holds in the case of $N$-expansive homeomorphism.

\subsection{DA-plug}\label{DA-plug} Before constructing our examples, it is necessary to recall the notion of DA-diffeomorphism in order to define what we shall call the \textit{unstable DA-plug} and \textit{stable DA-plug}, which will serve as fundamental components of our constructions. A Derived from Anosov diffeomorphism (denoted $f_{DA}$, or simply DA-diffeomorphism) is a diffeomorphism that exhibits dynamics similar to those of an Anosov system, but it is typically defined on manifolds where no genuine Anosov diffeomorphism exists.

More precisely, a DA-diffeomorphism is obtained by modifying an Anosov diffeomorphism in a small region of the manifold, thus producing a system that preserves many of the hyperbolic features, such as the existence of stable and unstable manifolds, but may introduce singularities or new types of invariant sets (for instance, saddle points).

One of the classical constructions consists in perturbing a linear Anosov diffeomorphism on the torus $\mathbb{T}^2$ to create a hyperbolic fixed point of saddle type while keeping the system structurally stable. These systems play an important role in the study of non-uniform hyperbolicity and robust dynamical properties. Further details regarding the construction can be found in Aoki-Hiraide \cite{aoki1994topological}, Katok-Hasselblatt \cite{katok1995introduction}, Palis-Melo \cite{palis1982geometric}, Robinson \cite{robinsonbook} and Shub \cite{shub1987global}.

\begin{teorema}[Robinson \cite{robinsonbook}\label{TeoremaRobinson}, $\S$ 7.8, p. $300$] The diffeomorphism $f_{DA}$ has non-wandering set $\Omega(f_{DA}) = \{p_0\} \cup \Lambda$, where $p_0$ is a source and $\Lambda$ is a hyperbolic expanding attactor of topological dimension one. The map $f_{DA}$ is transitive on $\Lambda$, and the periodic points are dense in $\Lambda$.
\end{teorema}

We present the following definition in order to construct a generalized version of the DA-diffeomorphism using what we call stable DA-plug (or unstable DA-plug).

\begin{definicao}
  Let  $f: \mathcal{U} \longrightarrow \mathbb{R}^2$ a $C^1$-diffeomorphism where 
  $\mathcal{U} := (-\epsilon, \epsilon) \times (-\epsilon, \epsilon)$.
 The pair $(\mathcal{U},f)$ is called a stable $DA$-plug (or unstable DA-plug respectively) if:
    \begin{enumerate}
        \item The decomposition $(\{x\} \times (-\epsilon,\epsilon))$ is $f$-invariant.
        \item The vertical line $\{0\} \times (-\epsilon, \epsilon)$ contains three hyperbolics fixed points:
        $p_0 = (0,0)$, which is a source (sink, respectively), and $p_1$, $p_2$  are saddles, which $p_1 < p_0 < p_2$. These are the only fixed points in $\mathcal{U}$.
        \item $f$ ($f^{-1}$, respectively) is a contraction on vertical lines $\{x\} \times (-\epsilon, \epsilon)$, $x \neq 0$.
    \end{enumerate}
\end{definicao}

The following theorem is a relaxation of the hypotheses in the construction of a DA-Anosov diffeomorphism, the proof essentially follows Robinson's construction [11],

\begin{teorema}\label{teorema2.2}
Let $f : \mathcal{U}=(-\epsilon, \epsilon) \times (-\epsilon, \epsilon)  \to \mathbb{R}^2 $ be a $ C^1 $-diffeomorphism with \( p_0 = (0,0) \) a hyperbolic fixed point of saddle type. Suppose that the decomposition into vertical lines \( \{x\} \times (-\epsilon, \epsilon) \) defines the leaves of the local stable foliation; that is, these vertical lines are invariant under \( f \), and \( f \) acts as a contraction along them. Then there exists a stable DA-plug \( (\mathcal{U}, \hat{f}) \) and there is a small disk \( D \subset \mathcal{U} \), centered at \( p_0 \) (source) and containing the two saddle points of \( \hat{f} \), with the property that $\hat{f}|_{\mathcal{U} \setminus D} = f|_{\mathcal{U} \setminus D}$.
\end{teorema}

\begin{proof} Let $f$ a $C^1$ diffeomorphism and $p_0 = (0,0)$ be a hyperbolic fixed point of saddle type ($\det (Df(p_0)) < 0$). We may assume that $\det(Df(p)) < 0$ for all $p \in \mathcal{U}$. 

By hypothesis, $\{0\} \times (-\epsilon, \epsilon)$ is the stable space of $p_0$, modulo a change of coordinates preserving the vertical lines we can assume that $ (-\epsilon, \epsilon)\times \{0\}$ is the unstable space of $p_0$. We denote by $v_0^u$ and $v_0^s$ the perpendicular eigenvectors in  $p_0$ while the horizontal segment generated by $v_0^u$ is the unstable manifold of $p_0$. 

Thus, we introduce the coordinate system $u_1 v_0^u + u_2v_0^s$ in $\mathcal{U} = (-\varepsilon, \varepsilon) \times (-\varepsilon, \varepsilon)$. Let $r_0 > 0$ such that $B_{r_0}(p_0) \subset U$ and $\delta : \mathbb{R} \xrightarrow{C^{\infty}} \mathbb{R}$ be a bump function such that $\delta(x) = 0$ for $|x| \geq r_0$, $\delta(x) = 1$ for $|x| \leq \frac{r_0}{2}$, and $\delta'(x) < 0$ for $\frac{r_0}{2} < |x| < y_0$. We define the vector field $V(u_1, u_2) = (0, u_2 \, \delta(\sqrt{u_1^2 + u_2^2}))$. Observe that $V \equiv 0$ outside $B_{r_0}(p_0)$ and along the horizontal line ($u_2 = 0$). Let $\varphi^t$ denote the \emph{complete} flow induced by vector field $V$. Clearly, $\varphi^t \equiv \text{id}$ outside $B_{r_0}(p_0)$ and along $\{u_2 = 0\}$. The diffeomorphism $f_t = \varphi^t \circ f$ preserves the vertical foliation and coincides with $f$ outside $B_{r_0}(p_0)$.

We have
\[
Df_t|_{(p_0)} = D\varphi^t_{(f(p_0))} \cdot Df_{(p_0)} = \begin{pmatrix} \lambda_0^u & 0 \\ 0 & e^t\lambda_0^s \end{pmatrix}
\]
where $0 < \lambda^s_0 < 1 < \lambda^u_0$ are the eigenvalues of $Df(p_0)$. Therefore, there exists $\tau > 0$ sufficiently large such that $e^{\tau} \lambda^s_0 > 1$. We define $\hat{f} = \varphi^{\tau} \circ f$. 

We note that the flow $\varphi^t$ preserves each vertical line, because of the form of the diﬀerential equations. Therefore, $\hat{f}$ preserves the vertical lines. 
Because $\hat{f}(p_0) = p_0$ is a source and outside $B_{r_0}(p_0)$
the slope of the graph of $\hat{f}$ on the vertical lines is negative, then $\hat{f}$ has three fixed points on  $\{0\} \times (-\epsilon, \epsilon)$, $p_0$ and two new fixed points $p_1$ and $p_2$ both of saddle type with $p_1 > p_0 > p_2$. Thus $(\mathcal{U},\hat{f})$ is a stable DA-plug.
\end{proof}

\begin{corolario}\label{corolário2.1}
Let \( f : \mathcal{U}\to \mathbb{R}^2 \) be a \( C^1 \)-diffeomorphism with \( p_0 = (0,0) \) a hyperbolic fixed point of saddle type. Suppose that the decomposition into vertical lines \( \{x\} \times (-\epsilon, \epsilon) \) defines the leaves of the local unstable foliation; that is, these vertical lines are invariant under \( f \), and \( f^{-1} \) acts as a contraction along them. Then there exists a unstable DA-plug \( (\mathcal{U}, \hat{f}) \) and there is a small disk \( D \subset \mathcal{U} \), centered at \( p_0 \) (sink) and containing the two saddle points of \( \hat{f} \), with the property that $\hat{f}|_{\mathcal{U} \setminus D} = f|_{\mathcal{U} \setminus D}$. 
\end{corolario} 

The proof simply involves constructing a stable DA-plug for $f^{-1}$, which will be an unstable DA-plug for $f$.

\enlargethispage{1\baselineskip}

\subsection{Gluing DA-plugs}\label{Gluing DA-plugs}

In this subsection we will proceed to glue two DA-plugs, always one stable and one unstable. We follow the construction presented in Artigue-Pacifico-Vieitez \cite{artiguepacificovieitez}. We need to repeat the fundamental steps of this construction because the notation and the way of schematizing will be necessary for future constructions.
 
Let $f_1: \mathcal{U}_1 \longrightarrow \mathbb{R}^2$ be  $C^1$-diffeomorphism  stable  $DA$-plug, and $f_2: \mathcal{U}_2 \longrightarrow \mathbb{R}^2$, a $C^1$-diffeomorphism  unstable  $DA$-plug. Note that $f_2$ is not necessarily conjugated to $f_1^{-1}$, it is only required that $(\mathcal{U}_2,f_2)$  be an unstable $DA$-plug. Let $p_0$ be the source of $f_1$ and $q_0$ be the sink of $f_2$. In both cases, the vertical lines are invariant under $f_1$ (resp. $f_2$), although $f_1$ does not contract these in the neighborhood of $p_1$, for future reasons, we will call this stable foliation for $f_1$ (resp. unstable foliation for $f_2$).
 
We will assume that there exist local charts $\varphi_i: D \longrightarrow \mathcal{U}_i$, $i=1,2$, with $D =\{x \in \mathbb{R}^2: ||x|| \le 2\} $, such that:

\begin{enumerate}
    \item $\varphi_1(0) = p_0$ and $\varphi_2(0) = q_0$,
    \item $ \varphi_1(D) \subset W^u(p_0)$ and $\varphi_2(D)  \subset W^s(q_0)$,
    \item The pull-back of the stable (res. unstable) foliation by $\varphi_1$ (resp. $\varphi_2$) is the vertical (resp. horizontal) foliation on $D$, and 
\item  $\varphi_1^{-1} \circ f_1^{-1} \circ \varphi_1(x) = \varphi_2^{-1} \circ f_2 \circ \varphi_2(x) = x/4$ for all $x \in D$.
\end{enumerate}

Let $A$ be the annulus $\{x \in \mathbb{R}^2: 1/2 \le ||x|| \le 2 \}$ and $\psi: \mathbb{R}^2 \longrightarrow \mathbb{R}^2$ the inversion $\psi(x) = x/||x||^2$. Consider the open disk $\hat{D}=\{x \in \mathbb{R}^2 \; : \; ||x|| < \frac{1}{2}\} $ (see Figure \ref{plugs}). On $[\mathcal{U}_1 \setminus \varphi_1(\hat{D})] \cup [\mathcal{U}_2 \setminus \varphi_2(\hat{D})]$ consider the equivalence relation generated by $\varphi_1(x) \sim \varphi_2 \circ \psi(x)$ for all $x \in A$. Denote by $\bar{x}$ the equivalence class of $x$. The set $ \mathcal{U}_1 \# \mathcal{U}_2 =[\mathcal{U}_1 \setminus \varphi_1(\hat{D})] \cup [\mathcal{U}_2 \setminus \varphi_2(\hat{D})] \; / \sim $ with the quotient topology is a surface, which we call connected sum. 
 
 Then we define homeomorphism  $\tilde{f}: \mathcal{U}_1 \# \mathcal{U}_2 \longrightarrow \mathcal{U}_1 \# \mathcal{U}_2$ by 

$$
\tilde{f}(x) = 
  \begin{cases}
      \overline{f_1(x)}, & x \in \mathcal{U}_1 \setminus \varphi_1(\hat{D}), \\
     \overline{f_2(x)} , & x \in \mathcal{U}_2 \setminus \varphi_2(\hat{D}). 
  \end{cases}
$$

\begin{figure}[ht]
    \centering
    \includegraphics[width=0.75\linewidth]{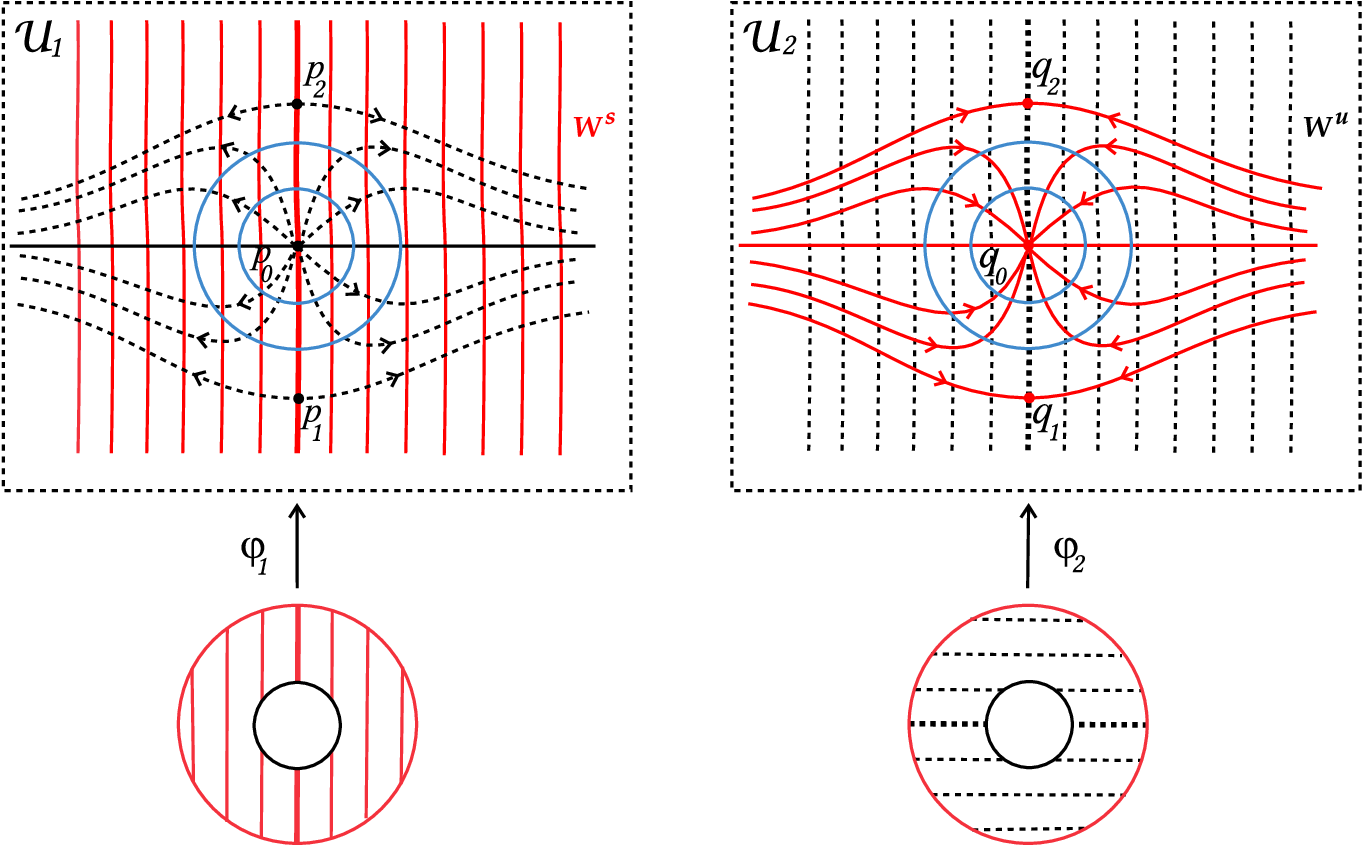}
    \caption{Stable and unstable DA-plugs.}
    \label{plugs}
\end{figure}

We can visualize, in  a Figure (\ref{anel}.a), the two foliations in the annulus A. The vertical lines (red) represent the stable foliation and the  dotted lines (black) represent the unstable foliation (after the inversion). In the Figure (\ref{anel}.b), we can see the surface $ \mathcal{U}_1 \# \mathcal{U}_2$. 

\begin{figure}[ht]
    \centering
    \includegraphics[width=0.73\linewidth]{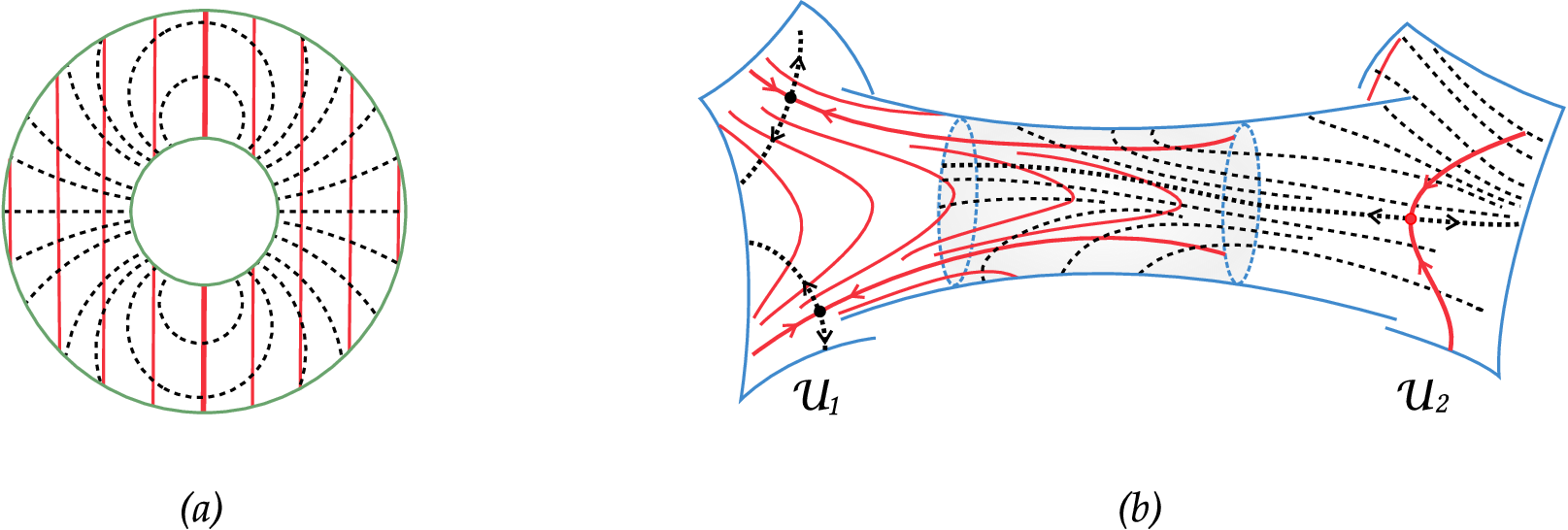}
    \caption{Annulus A and surface $ \mathcal{U}_1 \# \mathcal{U}_2$.}
    \label{anel}
\end{figure}

\subsection{2-expansive on bitorus}\label{2-expansive}
We need to detail the construction of the Artigue-Pacifico-Vieitez example \cite{artiguepacificovieitez} and recall some properties about its invariant sets. Let $S_i$ ($i = 1, 2)$ be disjoint copies of $\mathbb{T}^2$, and let $f_i: S_i \longrightarrow S_i$ be DA-Anosov diffeomorphisms with fixed points $f_1(p_0)=p_0$ and $f_2(q_0)=q_0$ (a source and a sink, respectively). Where, $f_1$ has a stable DA-Plug and $f_2$ has an unstable DA-Plug.

Let $V_1\subset W^u(p_0,f_1)$ (resp. $V_2\subset W^s(q_0,f_2)$) open sets, then we have 
$$ W^u(p_0,f_1)= \bigcup_{k=0}^{\infty} f_1^k(V_1)\;\;\;\; \mbox{and} \;\;\;\;W^s(q_0,f_2)= \bigcup_{k=0}^{\infty} f_2^{-k}(V_2) .$$  

Both are open and dense in $S_1$ and $S_2$, respectively, which we call \textit{gaps}. Then 

 $$\Lambda_1 = \bigcap_{k \geq 0} f_1^{k}\left(N_1  \right) \;\;\;\; \mbox{and} \;\;\;\; \Lambda_2 = \bigcap_{k \geq 0} f_2^{-k}\left(N_2 \right),$$ 

 where $N_i=S_i\setminus V_i$ ($i=1,2$) are trapping regions.

\newpage

The maps $f_i$ have a hyperbolic structure on $\Lambda_i$ respectively, then the set $\Lambda_1$ (resp. $\Lambda_2$ ) is a hyperbolic expanding attractor of topological dimension one to $f_1$ (resp. $f_2^{-1}$). The sets $\Lambda_i$ are transitive. The periodic points of $f_i$ are dense in $\Lambda_i$. 

The stable DA-plug of $f_1$ has two fixed points  $p_1$ and $p_2$ both of saddle type with $p_1>p_0>p_2$.  The unstable DA-plug of $f_2$ has two fixed points  $q_1$ and $q_2$ both of saddle type with $q_1>q_0>q_2$. We have for $j=1,2$, $W^u(p_j,f_1)$ is dense in $\Lambda_1$  and $W^s(q_j,f_2)$ is dense in $\Lambda_2$. For all $x\in \Lambda_1$, $W^s(x,f_1)$ is dense in $S_1$  and for all $x\in \Lambda_2$, $W^u(x,f_2)$ is dense in $S_2$. 

We now perform the connected sum of the stable and unstable plugs of $(S_1,f_1)$ and $(S_2,f_2)$, obtaining the bitorus surface $\Sigma_2=S_1 \# S_2$ and the homeomorphism $\displaystyle f:\Sigma_2 \longrightarrow \Sigma_2$. The homeomophism $f$ is $2$-expansive, but not expansive (see Artigue-Pacifico-Vietz \cite{artiguepacificovieitez}, Proposition 6.1). To facilitate development, we will henceforth use the position of the component of $S_1$ in $\Sigma_2$ with its attractor $\Lambda_1$ on the left, and consequently the repulsor $\Lambda_2$ on the right (see Figure \ref{colagemgenero3}).

\begin{figure}[ht]
    \centering
    \includegraphics[width=0.65\linewidth]{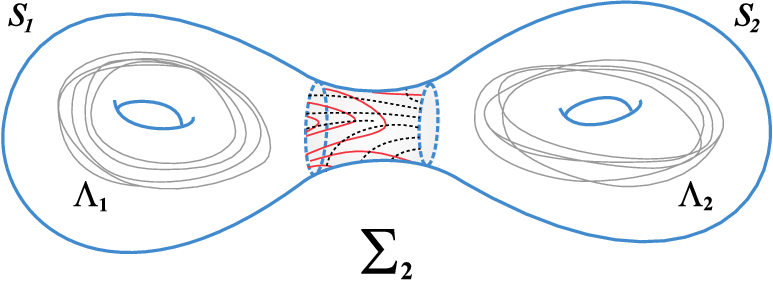}
    \caption{$\Sigma_2$, $\Lambda_1$ and $\Lambda_2$.}
    \label{colagemgenero3}
\end{figure}

\section{Proof of the theorems}\label{secaomainresults}

In this section, we present proofs of the Theorems (\ref{2expansivos}) and (\ref{thprincipal}) stated in the introduction. In fact, we construct families of new examples on compact surfaces. 

\begin{teorema}\label{2expansivos}
There exist 2-expansive homeomorphisms on compact surfaces of genus greater than or equal to 2 which are not expansive.
\end{teorema}

\begin{proof} Let $f:\Sigma_2 \longrightarrow \Sigma_2$ be a 2-expansive homeomorphism as constructed in Subsection (\ref{2-expansive}). That is, $\Sigma_2=S_1 \# S_2$ and $f \coloneqq f_1 \# f_2$, where $\Lambda_1 \subset S_1$ is a hyperbolic attractor whose periodic points are dense. We identify $S_1$ with the torus portrayed on the left of Figure (\ref{colagemgenero3}). Moreover, the stable set $W^s(\Lambda_1)$ is dense in $S_1$.

Thus, there exists $x_0 \in \Lambda_1$ such that $x_0$ is a periodic point of saddle type with period $k_0 \in \mathbb{Z}$ of $f_1$ and $W^u(x_0,f_1)$ is dense on $\Lambda_1$. Then, from the Lemma (\ref{2expansivos}) and observation, $f_1^{k_0}$ is a 2-expansive with $x_0$ fixed point of saddle type. For simplicity of notation, let's denote it by $f_1=f_1^{k_0}$, thus $x_0$ is a fixed point of saddle type for $f_1$ and modulo the change of coordinates, satisfying the hypotheses of the Theorem (\ref{teorema2.2}), therefore $f_1$ on $S_1$ has two stable DA-plugs, one centered on $p_0$ and the other on $x_0$.

Let $U_1\subset W^u(p_0,f_1)$ and $U_2\subset W^u(x_0,f_1)$, then $\bigcup_{i\geq 0}f^i(U_1 \cup U_2)$ give rise to two disjoint gaps in $S_1$. Denoting by $N=S_1\setminus (U_1 \cup U_2)$ is a trapping region and
$$\bigcap_{i\geq 0} f^i(N)=\widehat{\Lambda}_1$$

The set $\widehat{\Lambda}_1$ is hyperbolic expanding attractor of topological dimension one to $f_1$, and periodic points of $f_1$ are dense in $\widehat{\Lambda}_1$ and $W^s(\widehat{\Lambda}_1)$ é dense  in $S_1$.

Let $f_i:S_i \longrightarrow S_i$, $i=2,3$, be two DA-anosov diffeomorphisms with unstable DA-plugs, from the Subsection (\ref{Gluing DA-plugs}). We proceed to glue one of these unstable DA-plugs onto each stable DA-plug of $S_1$. Thus we have the homeomorphism $g:\Sigma_3 \longleftrightarrow \Sigma_3$ where $\Sigma_3$ is a surface compact of genus 3 and $g$ is a 2-expansive homeomorphism. Note that the decomposition of the stable and unstable sets into $\Sigma_3$ for $g$ is generated by iterations of the decomposition in the gluing annulus.

\begin{figure}[ht]
    \centering
    \includegraphics[width=0.75\linewidth]{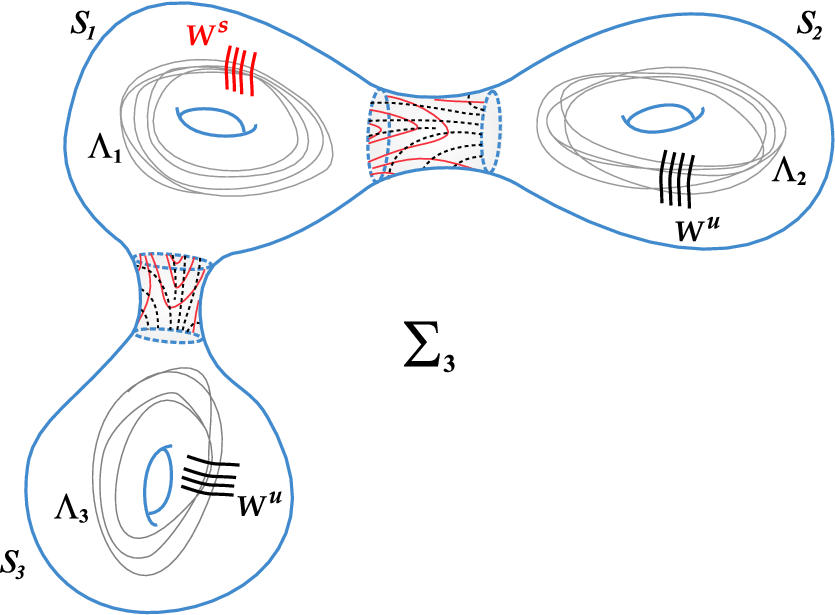}
    \caption{Genus 3 surface.}
    \label{fig:placeholder}
\end{figure}

Inductively, we can take a copy of $S_i$, $i=1,2,3$ in the construction above, search for a periodic point on the hyperbolic set $\Lambda_i$, if such a periodic point is in $\widehat{\Lambda}_1\subset S_1$, we proceed as above to obtain a third stable DA-plug, thus constructing an example in $\Sigma_4$, a surface compact of genus 4. If such a periodic point is chosen in $S_2$ or $S_3$, we proceed as above, but this time we obtain an unstable DA-plug; for this we must take a stable DA-plug to paste, thus obtaining an example in $\Sigma_4$.
\end{proof}

For the next proposition, we will use a pseudo-Anosov diffeomorphism with 1-prongs on the sphere. For this purpose, we need to construct it and discuss some of its properties. This example appears in P. Walters \cite{walters} (Example $1$, p. $140$). In \cite{dentritations} (Example 2.2.1), Artigue shows that it is $cw2$-expansive and that, for all $N\in \mathbb N$, it is not $N$-expansive.

Consider the 2-torus $\mathbb{T}^2=\mathbb{R}^2/\mathbb{Z}^2$, and fix the fundamental domain of $\mathbb{T}^2$ to be the square $[-1/2,1/2]\times [-1/2,1/2]$. Consider the equivalence relation $X \sim -X$ for $X\in \mathbb{T}^2$. The quocient space is a two-dimensional sphere $\mathbb{S}^2=\mathbb{T}^2/\sim$. Denote by $\Pi: \mathbb{T}^2 \longrightarrow \mathbb{S}^2$ the canonical projection. Let $f_A:\mathbb{T}^2 \longrightarrow \mathbb{T}^2 $ be a Linear Anosov diffeomorphism induced by $A\in SL(2,\mathbb{Z})$. Moreover the relation $X \sim -X$ is compatible with the Anosov map $f_A$, i. e.: $f_A(X)\sim -f_A(X)= f_A(-X)$ by linearity, and therefore projects to $\mathbb{S}^2$ called \textit{generalized pseudo-Anosov}  map $f:\mathbb{S}^2 \longrightarrow \mathbb{S}^2 $, where $f(x)=\Pi(f_A(\Pi^{-1}(x)))$. Observe that the projection $\Pi: \mathbb{T}^2 \longrightarrow \mathbb{S}^2$  is a branched covering. 

\newpage

Denote by $W^s$ and $W^u$ the stable and unstable singular foliations of $f$ respectively. These are transverse foliations except at the singularities. Singularities points are 1-prongs (that is, point $p\in \mathbb{S}^2$ in which the local stable and unstable are arcs finishing at $p$), for example the point $\Pi(0)$ is a point with a 1 prong. The foliations $W^s$ and $W^u$ in a neighborhood of the 1 prong $\Pi(0) \in \mathbb{S}^2 $ determine a regular bi-asymtotic sector (see Arruda-Carvalho-Sarmiento \cite{arrudacarvalhosarmiento}, definition 3.1), it is the foliations look as in Figure (\ref{bi-asymptotic}). Periodic points of $f$ are dense in $\mathbb{S}^2 $. If $p\in \mathbb{S}^2 $ is a periodic point, the $W^s(p,f)$ and $W^u(p,f)$ are dense in $\mathbb{S}^2 $.

\begin{figure}[ht]
    \centering
    \includegraphics[width=0.5\linewidth]{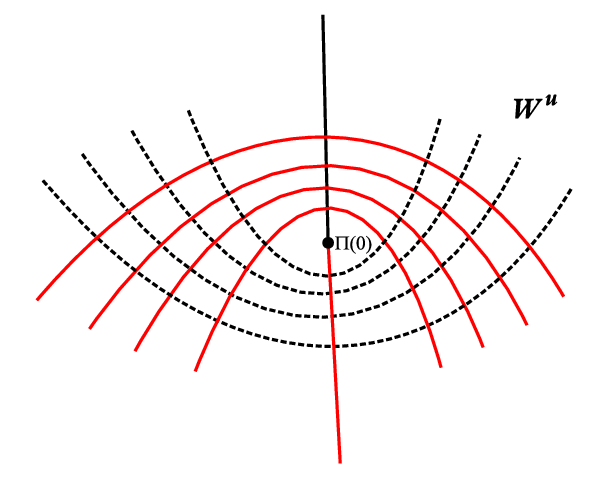}
    \caption{Regular bi-asyntotic sector.}
    \label{bi-asymptotic}
\end{figure}

Note that, as mentioned by Artigue in \cite{artiguerobustly} in the proof of Theorem 4.1, all homeomorphisms $f:\Sigma_g \longrightarrow \Sigma_g$ constructed in the previous  Theorem (\ref{2expansivos}) are $C^{\infty}$-diffeomorphisms. Therefore, any of these $C^{\infty}$-diffeomorphisms on a compact surface of genus $g\geq 2$ are Axiom $A$, because the non-wandering set consists of a hyperbolic attractor and a hyperbolic repeller, and has no cycles. Fixed in $\Sigma_g$ a gluing annulus $\overline{A}$. The stable and unstable foliations of $f$ in the annulus look like those shown in Figure 1 of Artigue  \cite{artiguerobustly}.

Consequently, by applying the same arguments used in the proof of Theorem 4.1 in  Artigue \cite{artiguerobustly}, we obtain the following theorem.

\begin{teorema}\label{robustez}
For each $r \geq 2$ there is a $C^r$-robustly $r$-expansive diffeomorphism that is not $(r-1)$-expansive on surface compact of genus $g \ge 2$ .
\end{teorema}

\begin{proposicao}\label{cw2expansivos-0}
There exist $cw2$-expansive homeomorphisms which are not transitive on the sphere $\mathbb{S}^2$ which are not $N$-expansive for all $N\geq 1$.
\end{proposicao}

\begin{proof}  Let $f$ a pseudo-Anosov map on $\mathbb{S}^2$, it is a $cw2$-expansive. Let $p_0\in \mathbb{S}^2$ periodic point of, say, period $k$, then of Lemma (\ref{iteradacw2expansivos}), $f_1:=f^k$ is a $cw2$-expansive and $f_1(p_0)=p_0$ and $W^s(p_0,f_1)$ and $W^u(p_0,f_1)$ are dense. The foliations $W^s(f_1)$ and $W^u(f_1)$ are transversal in a neighborhood of $p_0$. From Theorem (\ref{teorema2.2}), we can insert a stable DA-Plug centered on $p_0$ (obviously, we can also insert an unstable DA-Plug centered on $p_0$). We can also denote by $f_1$ the modified pseudo-Anosov set with a stable DA-plug; in this situation, $p_0$ is a source.We denote by $f_2$ another copy of the modified pseudo-Anosov map on $\mathbb{S}^2$ with unstable DA-plug centered on $q_0$; it is a sink. We proceed to glue $f_1$ and $f_2$ as in the Subsection (\ref{Gluing DA-plugs}), note that $\mathbb{S}^2\# \mathbb{S}^2=\mathbb{S}^2$, thus we obtain $f=f_1\# f_2: \mathbb{S}^2 \longrightarrow \mathbb{S}^2$ with the desired properties.
\end{proof}

\newpage

\begin{proposicao}\label{cw2expansivos-1}
There exist cw2-expansive homeomorphisms with a 1-prong on the torus $\mathbb{T}^2$ which are not $N$-expansive  for all $N\geq 1$.
\end{proposicao}

\begin{proof} To construct these examples, we take the modified pseudo-Anosov with a stable DA-plug on the sphere $\mathbb{S}^2$ as in the proof of the previous proposition, denoted by $f_1$. On the other hand, let $f_2$ be a $DA$-Anosov with an unstable plug on the torus $\mathbb{T}^2$. Since $\mathbb{S}^2 \# \mathbb{T}^2= \mathbb{T}^2$, we can glue the plugs together, thus we obtain $ f:=f_1 \# f_2 :\mathbb{T}^2 \longrightarrow \mathbb{T}^2$ is a $cw2$-expansive homeomorphism with 1-prongs.
\end{proof}

The modifications mentioned above, introduced by Artigue in \cite{artiguerobustly} on the plug gluing annulus yield an $r$-expansive homeomorphism that is not $(r-1)$-expansive on a compact surface (Theorem \ref{robustez}). These can now be applied to examples on the sphere (Proposition \ref{cw2expansivos-0}) or on examples on the torus (Proposition \ref{cw2expansivos-1}). Consequently, we obtain new examples of $cwN$-expansive homeomorphisms which are not cw(N-1)-expansive on the sphere and torus.

\subsection{Proof of Theorem \ref{thprincipal}}

As initially indicated, the construction of the annomalus examples follows the construction of Artigue \cite{artigueanomalous}. Next, we will describe the fundamental steps of Artigue's annomalus plug, and the proof basically shows that we can always add this plug to previously constructed examples.

Let be a linear transformation $T_1: \mathbb{R}^2\longrightarrow \mathbb{R}^2$, $\displaystyle T_1(x,y)=\left(\frac{x}{2}, \frac{y}{2}\right)$. This defines the non-locally connected continuum $E$:

\vspace{0.5cm}

$C(a) = \{(a,y) \in \mathbb{R}^2 : 0 \leq y \leq a\} \ \text{ for } a > 0,$

$D_1 = \bigcup_{i=1}^{\infty} C\!\left(\tfrac{1}{2} + \tfrac{1}{2^i}\right),$

$D_{n+1} = T_1(D_n) \quad \text{for all } n \geq 1,$

$D = \bigcup_{n \geq 1} D_n.$

\vspace{0.5cm}

The set $E = D \cup ([0,1] \times \{0\})$ is \textit{connected but not locally connected} (see Figure 2 in Artigue \cite{artigueanomalous}). To insert the set $E$ into the dynamics of a plane homeomorphism, the piecewise linear transformation $T:\mathbb{R}^2 \longrightarrow \mathbb{R}^2$ is defined.

$$T(x,y) = \left\{%
\begin{array}{lcl}
(\frac{x}{2}, \frac{y}{2}) & if  & x\geq y\geq 0, \\
(\frac{x}{2}, 2y) & if  & x\leq 0~\text{or}~y\geq 0, \\
(\frac{x}{2}, \frac{4y-3x}{2}) & if  & y\geq x \geq 0. \\
\end{array}%
\right.$$

The vertical vector field $X(p) = (0, \rho(p))$ is defined, where $\rho(p):=dist(p,E)$ denotes the distance from $p$ to $E$. The vector field $X$ has a complete flow $\phi:\mathbb{R} \times \mathbb{R}^2 \longrightarrow \mathbb{R}^2 $.

Let $g: [0,1] \times \mathbb{R} \longrightarrow [0,1] \times \mathbb{R}$ be the homeomorphism

$$g(x,y) = \left\{%
\begin{array}{lcl}
\phi^1\circ T(x,y) & if  & y\geq 0, \\
T(x,y) & if  & y<0. \\
\end{array}%
\right. ,$$
where $\phi^1$ is a time one homeomorphism associated to the vector field $X$.

The homeomorphism $g$  commutes with $T$ and preserves the vertical foliation on $\mathbb{R}^2$. The stable set of the origem $W^s(0,g)$ satifies 

$$W^s_g(0) ~ \cap ~([0,1] \times [0,1]) =  E.$$ 

Let us consider the rectangle ${\cal Q}\subset \mathbb{R}^2$ of vertices $(0,-4)$, $(-2,-4)$, $(2,4)$ and $(0,4)$, we note that $g({\cal Q})$ is a trapeze of vertices $(0,-8)$, $(-1,-8)$, $(1,5)$ and $(0.8)$. Let's call ${\cal Q}\cup g({\cal Q})$ anomalous plug, see Figure (\ref{anomalous plug}.a).

\begin{teorema} \label{thprincipal}
 Every compact surface ($g \ge 0$) admits a cw-expansive homeomorphism with a connected stable set but non-locally connected.
\end{teorema} 

\begin{proof} 
In any of the examples $(\Sigma_k,f)$, $k\geq 0$, of the Propositions (\ref{2expansivos}), (\ref{cw2expansivos-0}) or (\ref{cw2expansivos-1}), we fix one of the tori (or sphere) with a stable DA-plug, let's denote this component of $\Sigma_k$ by $(S_1,f_1)$, with the notation fixed in these examples we have that in the stable DA-plug two fixed points $p_1$ and $p_2$ of saddle type were inserted and $p_0$ a source, let's fix one of the saddles, say $p_1$, then one of the components $W^s(p_1, f_1)\subset W^u(p_0, f_1)$ connects from $p_0$ to $p_1$.

Denote by ${\cal R}_{p_1}\subset S_1$ a topological rectangle illustrated in Figure (\ref{anomalous plug}.b). Then just replace $ {\cal R}_{p_1}$ with anomalous plug ${\cal Q}$. For this purpose, consider the homeomorphism $h: {\cal R}_{p_1} \cup f({\cal R}_{p_1})\rightarrow
{\cal Q} \cup g({\cal Q})$ such that $f_1(x) = h^{-1} \circ g \circ h(x)$ for all $x \in \partial {\cal R}_{p_1}$. Defines the homeomorphism $F_1: S_1 \rightarrow S_1$ by

$$
F_1(x) = 
  \begin{cases}
    f_1(x), & \text{if} \; x \notin \mathcal{R}_{p_1}, \\
    h^{-1} \circ g \circ h(x) , & \text{if} \; x \in \mathcal{R}_{p_1}. 
  \end{cases}
$$

In this way we obtain a homeomorfism $F_1$ that we can call {\em stable DA-plug with an anomalous saddle}. 

Now, using the Proposition (\ref{Gluing DA-plugs}), we proceed to glue the corresponding components of $\Sigma_k$ onto $(S_1,F_1)$. This gives us the homeomorphism $f_k:\Sigma_k \longrightarrow \Sigma_k$. Let $A$ denote the gluing annulus of the stable DA-plug with an anomalous saddle of $S_1$ with an unstable DA-plug, say, $S_2 \setminus D \subset \Sigma_k$ ($D$ disk centered on the fixed point $q_0$ attractor, removed from $S_2$ to perform the gluing of plugs).

In this construction we can see that the dynamics of the homeomorphism $f_k$ on the components of $\Sigma_k$, with the exception of the component corresponding to $S_2$, are the same as those of $f$, that is, the non-wandering sets are expansive and dynamically isolated. In the component corresponding to $S_2$ we have the unstable DA-plug that will bond with the stable DA-plug with an anomalous saddle of $S_1$, eventually $S_2$ may have a finite number of other unstable DA-plugs, thus $S_2$ is invariant under iterations $f_k^{-i}$, $i\geq 0$. Thus we still have set $\widehat{\Lambda}_2 \subset S_2 $ which is expanding attractor of topological dimension one to $f_k^{-1}$ and $W^u(\widehat{\Lambda}_2)$ is dense in $S_2$, we know from the gluing of plugs that this foliation restricted to the annulus $A$ its unstable leaves are circle arcs, as Figure (\ref{anel}.a). The stable sets make a cw-foliations (see Artigue \cite{dentritations}) as Figure (\ref{anel}.b). The set $E$, non-locally connected, converges to the fixed point $p_1$ for future iterations; for negative iterations, it will enter a gap at $S_2$, becoming stuck at $S_2$.

As was well pointed out in Artigue's proof of theorem 2.5 in \cite{artigueanomalous}, we may eventually have that a stable arc in $A\setminus E$ may contain a circle arc of the unstable foliations. The same arguments from Artigue's proof apply in this situation, so we can denote by $f_k:\Sigma_k \longrightarrow \Sigma_k$ the $cw$-expansive homeomorphism having a fixed point whose local stable set is connected but it is not locally connected.
\end{proof} 

\begin{figure}[ht]
  \centering
   \includegraphics[width=0.75\linewidth]{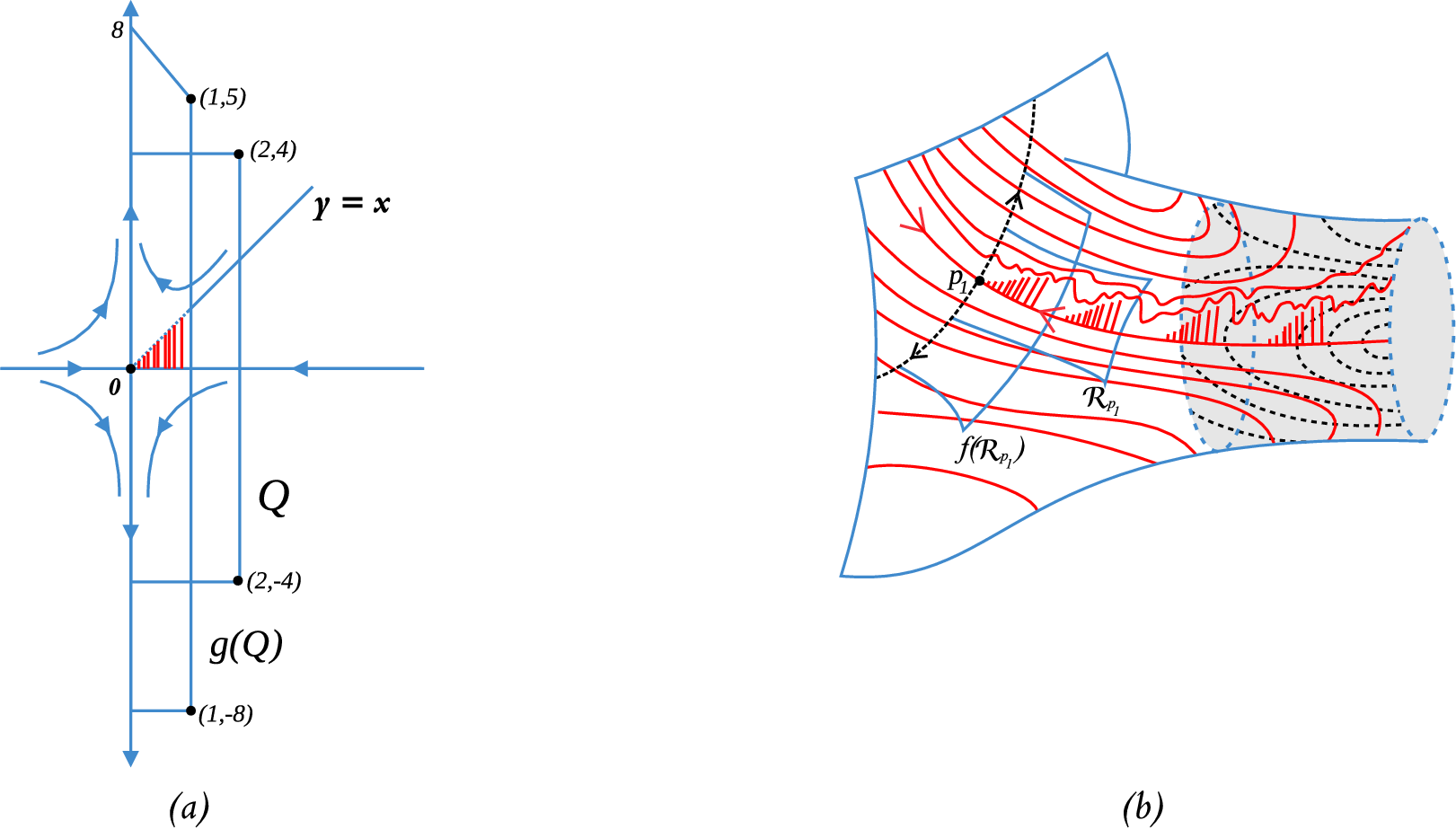}
   \caption{Anomalous plug.}
   \label{anomalous plug}
\end{figure}

\newpage

\noindent\textbf{Acknowledgments:} \noindent We are grateful to Rodrigo Arruda for his careful reading and comments.

\bibliography{bibliografia}

\vspace{2cm }

\noindent \textit{A. Sarmiento}\\
Departamento de Matemática\\
Universidade Federal de Minas Gerais - UFMG\\
Av. Antônio Carlos, 6627 - Campus Pampulha\\
Belo Horizonte - MG, Brasil\\
e-mail: \textit{sarmiento@mat.ufmg.br}

\vspace{1cm}

\noindent \textit{D. Danton}\\
Instituto Federal do Norte de Minas Gerais - IFNMG\\
Rodovia BR 367, Km 07, s/n - Zona Rural\\
Almenara - MG, Brasil\\
e-mail: \textit{douglas.nepomuceno@ifnmg.edu.br}

\vspace{1cm}

\noindent \textit{V. Valério}\\
Departamento de Matemática e Estatística\\
Universidade Federal de São João del-Rei - UFSJ\\
Praça Frei Orlando, 170, Centro\\
São João del-Rei - MG, Brasil\\
e-mail: \textit{vivipardini@ufsj.edu.br}

\end{document}